\documentclass[11pt]{article}
\usepackage{amsthm}
\usepackage{latexsym,amssymb,amsmath}
\usepackage{graphicx,float,color,fancybox,shapepar,setspace,hyperref}
\usepackage{subfigure}
\usepackage{mathrsfs}
\usepackage{cite}
\usepackage{tabularx}
\usepackage{pgf,tikz}
\usetikzlibrary{arrows}
\usetikzlibrary{shapes.geometric}
\newtheorem{thm}{Theorem}[section]
\newtheorem{case}[]{Case}
\newtheorem{subcase}[]{Subcase}[case]
\newtheorem{Subcase}[]{Subcase}[subcase]

\newtheorem{claim}[]{Claim}

\def\dfn#1{{\sl #1}}

\begin{document} 
\title{Gallai-Ramsey numbers for graphs with five vertices\\
	 of chromatic number four}
	
\author{Qinghong Zhao\thanks{Corresponding Author. Supported by the Summer Graduate Research Assistantship Program of Graduate School. E-mail address: qzhao1@olemiss.edu (Q. Zhao).}~~and~Bing Wei\thanks{Supported in part by the summer faculty research grant from CLA at the University of Mississippi. E-mail address: bwei@olemiss.edu (B. Wei).}\\ 
\small Department of Mathematics, University of Mississippi, University, MS 38677,  USA}
	
	\date{}
	\maketitle
\begin{abstract}
Given a graph $H$ and an integer $k\ge1$, the Gallai-Ramsey number $GR_k(H)$ is defined to be the minimum integer $n$ such that every $k$-edge coloring of $K_n$ contains either a rainbow (all different colored) triangle or a monochromatic copy of $H$. In this paper, we determine the Gallai-Ramsey numbers for connected graphs with five vertices of chromatic number four.\\
		
\noindent{\bf Key words}: Gallai coloring, Gallai-Ramsey number, Rainbow triangle\\
\noindent{\bf 2010 Mathematics Subject Classification}: 05C55; 05D10; 05C15
\end{abstract}
	
\section{Introduction}
In this paper, we only deal with finite, simple and undirected graphs. Given a graph $G$ and the vertex set $V(G)$, let $|G|$ denote the number of vertices of $G$ and $G[W]$ denote the subgraph of $G$ induced by a set $W\subseteq V(G)$. Given disjoint vertex sets $X,Y\subseteq V(G)$, if each vertex in $X$ is adjacent to all vertices in $Y$ and all the edges between $X$ and $Y$ are colored with the same color, then we say that $X$ is \dfn{$mc$-adjacent} to $Y$, that is, $X$ is \dfn{blue-adjacent} to $Y$ if all the edges between $X$ and $Y$ are colored with blue. We use $P_n$ and $K_n$ to denote the path and complete graph on $n$ vertices, respectively. For an integer $k\ge1$, we define $[k]=\{1,\ldots,k\}$.\medskip
	
The complete graphs under edge coloring without rainbow triangle usually have pretty interesting and somehow beautiful structures. In 1967, Gallai studied the structure under the guise of transitive orientations and obtained the following result \cite{Gallai} which was restated in \cite{Gy} in the terminology of graphs. 
\begin{thm}[\cite{Gallai},\cite{Gy}]\label{Gallai}
For a complete graph $G$ under any edge coloring without rainbow triangle, there exists a partition of vertices (called a Gallai-partition) with parts $V_1, V_2, \dots, V_\ell$, $\ell\ge2$, such that there are at most two colors on the edges between the parts and only one color on the edges between each pair of parts.
\end{thm}

We define $\mathcal{G}=G[\{v_1,\dots,v_\ell\}]=K_{\ell}$ as a \dfn{reduced graph} of $G$, where $v_i\in V_i$ for $i\in[\ell]$. Obviously, there exists a monochromatic copy of $H$ in $\mathcal{G}$ if $\ell\ge R_2(H)$, which leads to a monochromatic copy of $H$ in $G$. In honor of Gallai's result, the edge coloring of a complete graph without rainbow triangle is called \dfn{Gallai coloring}. We use $(G,c)$ to denote a complete graph $G$ under the Gallai coloring $c: E(G)\xrightarrow{} [k]$. Given graphs $H_1,\ldots, H_k$ and an integer $k \ge 1$, the \dfn{Ramsey number} $R(H_1,\ldots, H_k)$ is defined to be the minimum integer $n$ such that every $k$-edge coloring of $K_n$ contains a monochromatic copy of $H_i$ in color $i\in [k]$, and the \dfn{Gallai-Ramsey number} $GR(H_1,\ldots,H_k)$ is defined to be the minimum integer $n$ such that every Gallai $k$-coloring of $K_n$ contains a monochromatic copy of $H_i$ in color $i\in [k]$. We simply write $R_k(H_1)$ and $GR_k(H_1)$ when $H_1=\dots=H_k$ and $R((k-r)H_{r+1},rH_1)$ and $GR_k((k-r)H_{r+1},rH_1)$ when $H_1=\dots=H_r$ and $H_{r+1}=\dots=H_k$. Similar to the notation $GR_k((k-s-r)H,sF,rK)$. Clearly, $GR_2(H,F)=R(H,F)$ and $GR_k(H) \leq R_k(H)$ for $k\ge1$. However, determining the exact value of $GR_k(H)$ for a graph $H$ is far from trivial, even for a small graph. The general behavior of $GR_k(H)$ for all graphs $H$ was established in \cite{exponential}.
\begin{thm}[\cite{exponential}]
Let $H$ be a fixed graph  with no isolated vertices. Then $GR_k(H)$ is exponential in $k$ if $H$ is not bipartite, linear in $k$ if $H$ is bipartite but not a star, and constant (does not depend on $k$) when $H$ is a star.		
\end{thm}
Fox, Grinshpun and Pach~\cite{FGP} posed a conjecture for $GR_k(K_t)$. The cases $t=3,4$ were proved in \cite{chgr},\cite{exponential} and \cite{k4}. Recently, there is a breakthrough for the case $t=5$ in \cite{colt} while there is no any results for the cases $t\ge 6$. For more recent work on Gallai-Ramsey numbers, we refer the readers to \cite{FGP,FMO}. It is known that there are twenty three isolated-free graphs with five vertices, denoted by $\mathscr{H}$. The Gallai-Ramsey numbers for the graphs in $\mathscr{H}$ with chromatic number two and three were completely determined in \cite{4,five,YM,YM1,W2N,HW,17,JZ,ZF,ZQ}. In this paper, we obtain the Gallai-Ramsey numbers for graphs in $\mathscr{H}$ with chromatic number four as follows. Clearly, there are three such graphs (see Fig. 1). 
\begin{center}
	\def\r{6pt}
	\def\dy{1cm}
	\tikzset{c/.style={draw,circle,fill=black,minimum size=\r,inner sep=0pt,
			anchor=center},
		d/.style={draw,circle,fill=black,minimum size=\r,inner sep=0pt, anchor=center}}	
	\scalebox{0.5}{
		\begin{tikzpicture}
		\pgfmathtruncatemacro{\Ncorners}{5}
		\node[ regular polygon,regular polygon sides=\Ncorners,minimum size=3cm] 
		(poly\Ncorners) {};
		\foreach\x in {1,...,\Ncorners}{
			\node[d] (poly\Ncorners-\x) at (poly\Ncorners.corner \x){};
		}	
		\foreach\X in {1,...,\Ncorners}{
			\ifnum\X=2
			\foreach\Y in {1,3,4,5}{
				\draw (poly\Ncorners-\X) -- (poly\Ncorners-\Y);}
			\fi				
			\ifnum\X=5
			\foreach\Y in {3,4}{
				\draw (poly\Ncorners-\X) -- (poly\Ncorners-\Y);}
			\fi
			\ifnum\X=3
			\foreach\Y in {4}{
				\draw (poly\Ncorners-\X) -- (poly\Ncorners-\Y);}
			\fi		
		}	
		\end{tikzpicture}
	}
	\hfil
	\scalebox{0.5}{
		\begin{tikzpicture}
		\pgfmathtruncatemacro{\Ncorners}{5}
		\node[ regular polygon,regular polygon sides=\Ncorners,minimum size=3cm] 
		(poly\Ncorners) {};
		\foreach\x in {1,...,\Ncorners}{
			\node[d] (poly\Ncorners-\x) at (poly\Ncorners.corner \x){};
		}	
		\foreach\X in {1,...,\Ncorners}{
			\ifnum\X=2
			\foreach\Y in {1,3,4,5}{
				\draw (poly\Ncorners-\X) -- (poly\Ncorners-\Y);}
			\fi			
			\ifnum\X=5
			\foreach\Y in {1,3,4}{
				\draw (poly\Ncorners-\X) -- (poly\Ncorners-\Y);}
			\fi
			\ifnum\X=3
			\foreach\Y in {4}{
				\draw (poly\Ncorners-\X) -- (poly\Ncorners-\Y);}
			\fi
		}	
		\end{tikzpicture}
	}
	\hfil
\scalebox{0.5}{
	\begin{tikzpicture}
	\pgfmathtruncatemacro{\Ncorners}{5}
	\node[ regular polygon,regular polygon sides=\Ncorners,minimum size=3cm] 
	(poly\Ncorners) {};
	\foreach\x in {1,...,\Ncorners}{
		\node[d] (poly\Ncorners-\x) at (poly\Ncorners.corner \x){};
	}	
	\foreach\X in {1,...,\Ncorners}{
		\ifnum\X=2
		\foreach\Y in {1,3,4,5}{
			\draw (poly\Ncorners-\X) -- (poly\Ncorners-\Y);}
		\fi				
		\ifnum\X=5
		\foreach\Y in {1,3,4}{
			\draw (poly\Ncorners-\X) -- (poly\Ncorners-\Y);}
		\fi
		\ifnum\X=1
		\foreach\Y in {3,4}{
			\draw (poly\Ncorners-\X) -- (poly\Ncorners-\Y);}
		\fi	
	}	
	\end{tikzpicture}
}	
\begin{tikzpicture}	
\node at (-4.6,0) {$H_1$};
\node at (0,0) {$H_2$};
\node at (4.6,0) {$H_3$};
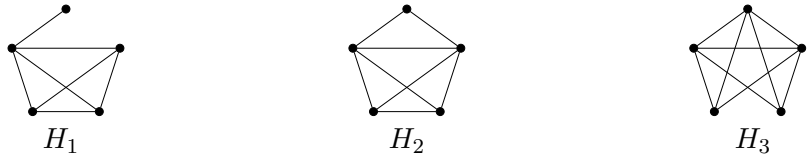
\node at (0,-.8){Fig. 1: The graphs with five vertices of chromatic number four.};
\end{tikzpicture}
\end{center}
\begin{thm}\label{t1}
Let $H\in\{H_1,H_2,H_3\}$. For any integer $k\ge1$, we have
\[GR_k(H)=\begin{cases}
(R_2(H)-1)\cdot17^{(k-2)/2}+1, & \textup{if $k$ is even,}\\
4\cdot17^{(k-1)/2}+1, & \textup{if $k$ is odd,}
\end{cases}\]
where $R_2(H_1)=R_2(H_2)=18$ and $R_2(H_3)=22$.
\end{thm}

In order to prove Theorem \ref{t1}, we actually prove the following two theorems. Note that Theorem \ref{t1} follows  from Theorems \ref{t1.3} and \ref{t1.4} by setting $r=k$ and $s=0$.
\begin{thm}\label{t1.3}
Let $H\in\{H_1,H_2\}$. For any integers $k\ge1$ and $r$ with $0\le r\le k$, 
\[GR_k((k-r)K_3,rH)=~~~~~~~~~~~~~~~~~~~~~~~~~~~~~~~~~~~~~~~~~~~~~~~~~~~~~~\]
\[\begin{cases}
5^{(k-r)/2}\cdot17^{r/2}+1, & \textup{if $(k-r)$ and $r$ are both even,~~~$(a_1)$}\\
2\cdot5^{(k-r-1)/2}\cdot 17^{r/2}+1, & \textup{if $(k-r)$ is odd and $r$ is even,~~~$(a_2)$}\\
8\cdot5^{(k-r-1)/2}\cdot17^{(r-1)/2}+1, & \textup{if $(k-r)$ and $r$ are both odd,~~~~$(a_3)$}\\
4\cdot5^{(k-r)/2}\cdot17^{(r-1)/2}+1, & \textup{if $(k-r)$ is even and $r$ is odd.~~~$(a_4)$}
\end{cases}\]	
\end{thm}
\begin{thm}\label{t1.4}
For any integers $k\ge1$, $s$ and $r$ with $0\le s\le k$ and $0\le r\le k$,
\[GR_k((k-s-r)P_3,sK_3,rH_3)=~~~~~~~~~~~~~~~~~~~~~~~~~~~~~~~~~~~~~~~~~~~~~~~~~~~~~\]
\[\begin{cases}
5^{s/2}\cdot\lfloor21\cdot17^{(r-2)/2}\rfloor+1, & \textup{if $s,r$ are both even and $s+r=k$, $(b_1)$}\\
2\cdot5^{s/2}\cdot17^{r/2}+1, & \textup{if $s,r$ are both even and $s+r<k$, $(b_2)$}\\
5^{(s-1)/2}\cdot\lfloor42\cdot17^{(r-2)/2}\rfloor+1, & \textup{if $s$ is odd, $r$ is even and $s+r=k$, $(b_3)$}\\
4\cdot5^{(s-1)/2}\cdot17^{r/2}+1, & \textup{if $s$ is odd, $r$ is even and $s+r<k$, $(b_4)$}\\ 
10\cdot5^{(s-1)/2}\cdot17^{(r-1)/2}+1, & \textup{if $s,r$ are both odd and $s+r=k$, ~$(b_5)$}\\
16\cdot5^{(s-1)/2}\cdot17^{(r-1)/2}+1, & \textup{if $s,r$ are both odd and $s+r<k$, ~$(b_6)$}\\
4\cdot5^{s/2}\cdot17^{(r-1)/2}+1, & \textup{if $s$ is even, $r$ is odd and $s+r=k$, $(b_7)$}\\
\lfloor32\cdot5^{(s-2)/2}\rfloor\cdot17^{(r-1)/2}+1, & \textup{if $s$ is even, $r$ is odd and $s+r<k$. $(b_8)$}
\end{cases}\]
\end{thm}

For ease of notations, let $GR_k((k-r)K_3,rH)=w(k,r)+1$ and $GR_k((k-s-r)P_3,sK_3,rH_3)=f(k,s,r)+1$. In this paper, we first show that $GR_k((k-r)K_3,rH)\ge w(k,r)+1$ and $GR_k((k-s-r)P_3,sK_3,rH_3)\ge f(k,s,r)+1$ by construction in Section 2 and then we prove that such lower bounds are actually the upper bounds in Section 3. 
\section{Preliminaries}
In this section, we first list some known results that shall be applied in the proof of Theorems \ref{t1.3} and \ref{t1.4}.
\begin{thm}[\cite{MC}]\label{MC}
$R(K_3,H_1)=R(K_3,H_2)=9$ and $R(K_3,H_3)=11$.
\end{thm}
\begin{thm}[\cite{GG}]\label{GG}
$R_2(K_3)=6$ and $R(K_3,K_4)=9$.
\end{thm}
\begin{thm}[\cite{R2}]\label{R2}
$R_2(H_1)=R_2(H_2)=18$ and $R_2(H_3)=22$.
\end{thm}
\begin{thm}[\cite{SPR}]\label{SPR}
$R(P_3,K_3)=5$ and $R(P_3,H_3)=7$.
\end{thm}
\begin{thm}[\cite{4}]\label{4}
For any integer $k\ge1$, $GR_k(P_3)=3$.	
\end{thm}
\begin{thm}[\cite{chgr,exponential}]\label{c} For any integer $k\ge1$,
	\[GR_k(K_3)=\begin{cases}
	5^{k/2}+1, & \textup{if $k$ is even,}\\
	2\cdot5^{(k-1)/2}+1, & \textup{if $k$ is odd.}
	\end{cases}\]
\end{thm}

We next give the general lower bounds of $GR_k((k-r)K_3,rH)$ with $H\in\{H_1,H_2\}$ and $GR_k((k-s-r)P_3,sK_3,rH_3)$ for all $k\ge1$ and $s,r$ with $0\le s\le k$ and $0\le r\le k$ by construction. For all constructions, we start with an $i$-edge colored complete graph $G_i$ (constructed below) and inductively suppose we have constructed an $i$-edge colored complete graph $G_i$ containing neither rainbow triangle nor appropriately colored monochromatic copy of $H\in\{H_1,H_2,H_3\}$, $K_3$ and $P_3$. For each two unused colors assigned to $H\in\{H_1,H_2,H_3\}$, we construct $G_{i+2}$ by replacing each vertex of $K_{17}$ with a copy of $G_i$ such that $K_{17}$ is colored with this two unused colors without monochromatic $K_4$. For each two unused colors assigned to $K_3$, we construct $G_{i+2}$ by replacing each vertex of $K_5$ with a copy of $G_i$ such that $K_5$ is colored with this two unused colors without monochromatic $K_3$. The base graphs for this construction are constructed as follows.\medskip 

We first construct the base graphs of Theorem \ref{t1.3}. For case $(a_1)$, the base graph $G_2$ is a 2-colored complete graph on $R_2(K_3)-1=5$ vertices without monochromatic copy of $K_3$. For case $(a_2)$, the base graph $G_1$ is a 1-colored complete graph on $|K_3|-1=2$ vertices without monochromatic copy of $K_3$. For case $(a_3)$, the base graph $G_2$ is a 2-colored complete graph on $R(K_3,H)-1=8$ vertices containing neither monochromatic copy of $H$ nor monochromatic copy of $K_3$. For case $(a_4)$, the base graph $G_1$ is a 1-colored complete graph on $|H|-1=4$ vertices without monochromatic copy of $H$.\medskip

We now construct the base graph of Theorem \ref{t1.4} based on the relations between $s+r$ and $k$. Assume that $s+r=k$. For case $(b_1)$, if $k=2$, then $G_2$ is a 2-colored complete graph on $R_2(H_3)-1=21$ vertices without monochromatic copy of $H_3$ when $s=0$, and $R_2(K_3)-1=5$ vertices without monochromatic copy of $K_3$ when $r=0$. Let $G_2$ be colored with colors 1 and 2 and $K_5$ be colored with colors 3 and 4 containing no monochromatic copy of $K_3$. Then the base graph $G_4$ is a 4-colored complete graph on 105 vertices obtained by replacing each vertex of $K_5$ with a copy of $G_2$ under $s=0$. For case $(b_3)$, if $k=1$, then $G_1$ is a 1-colored complete graph on $|K_3|-1=2$ vertices with color 1 containing no monochromatic copy of $K_3$. Let $G^{'}$ be a 2-colored complete graph on $R_2(H_3)-1=21$ vertices with colors 2 and 3 containing no monochromatic copy of $H_3$. Then the base graph $G_3$ is a 3-colored complete graph on 42 vertices obtained by replacing each vertex of $G_1$ with a copy of $G^{'}$. For case $(b_5)$, the base graph $G_2$ is a 2-colored complete graph on $R(K_3,H_3)-1=10$ vertices containing neither monochromatic copy of $H_3$ nor monochromatic copy of $K_3$. For case $(b_7)$, the base graph $G_1$ is a 1-colored complete graph on $|H_3|-1=4$ vertices without monochromatic copy of $H_3$.\medskip

Assume that $s+r<k$. For case $(b_2)$, the base graph $G_{k-s-r}$ is a $(k-s-r)$-colored complete graph on $GR_{k-s-r}(P_3)-1=2$ vertices with colors in $[k]\setminus[s+r]$ containing no monochromatic copy of $P_3$. For case $(b_4)$, the base graph $G_{k-(s-1)-r}$ is a $(k-s-r+1)$-colored complete graph on 4 vertices obtained by joining two copies of $G_{k-s-r}$ such that all edges between two copies are colored with the color assigned to $K_3$. For case $(b_6)$, let $G^{''}$ be a 2-colored complete graph on $R(K_4,K_3)-1=8$ vertices such that $G^{''}$ contains neither monochromatic copy of $K_4$ in color assigned to $H_3$ nor monochromatic copy of $K_3$. Then the base graph $G_{k-(s-1)-(r-1)}$ is a $(k-s-r+2)$-colored complete graph on 16 vertices obtained by replacing each vertex of $G^{''}$ with a copy of $G_{k-s-r}$. For case $(b_8)$, we first suppose $s=0$. Let $K_3$ be a monochromatic graph with the color assigned to $H_3$. Then the base graph $G_{k-s-(r-1)}$ is a $(k-s-r+1)$-colored complete graph on 6 vertices obtained by replacing each vertex of $K_3$ with a copy of $G_{k-s-r}$. We next suppose $s\ge2$. Then the base graph $G_{k-(s-2)-(r-1)}$ is a $(k-s-r+3)$-colored complete graph on 32 vertices obtained by joining two copies of the base graph of case $(b_6)$ such that all edges between two copies are colored with the new color assigned to $K_3$.\medskip

It is easy to check that all the base graphs do not contain rainbow triangle and thus $G_{i+2}$ is the desired construction containing neither rainbow triangle nor appropriately colored monochromatic copy of $H\in\{H_1,H_2,H_3\}$, $K_3$ and $P_3$. Therefore, we see that $GR_k((k-r)K_3,rH)\ge w(k,r)+1$ with $H\in\{H_1,H_2\}$ and $GR_k((k-s-r)P_3,sK_3,rH_3)\ge f(k,s,r)+1$ for all $k\ge1$ and $s,r$ with $0\le s\le k$ and $0\le r\le k$.\medskip

We now present two tables that consist of the ratios of corresponding functions to $w(k,r)$ and $f(k,s,r)$ based on the cases $a_1$-$a_4$ and $b_1$-$b_8$. We refer readers to \textbf{Appendix} and Theorems \ref{t1.3} and \ref{t1.4} for the specific formulas of listed functions in both tables.\medskip

\begin{center}
\begin{tabular}{ | m{1.5cm} | m{.8cm}| m{.8cm} | m{.8cm} | m{.8cm} |} 
\hline
Case&  $~~a_1$ & $~~a_2$ & $~~a_3$ & $~~a_4$\\ 
\hline
$~~$  $\frac{w(k-1,r)}{w(k,r)}$  & $~~\frac{2}{5}$ & $~~\frac{1}{2}$ & $~~\frac{1}{2}$ & $~~\frac{2}{5}$\\ 
\hline
$~~$  $\frac{w(k-2,r)}{w(k,r)}$ & $~~\frac{1}{5}$ & $~~\frac{1}{5}$ & $~~\frac{1}{5}$ & $~~\frac{1}{5}$\\ 
\hline
$~~$  $\frac{w(k,r-1)}{w(k,r)}$ & $~~\frac{8}{17}$ & $~~\frac{10}{17}$ & $~~\frac{5}{8}$ & $~~\frac{1}{2}$\\ 
\hline
$~~$  $\frac{w(k-1,r-1)}{w(k,r)}$ & $~~\frac{4}{17}$ & $~~\frac{4}{17}$ & $~~\frac{1}{4}$ & $~~\frac{1}{4}$\\
\hline
$~$  $\frac{w(k-2,r-1)}{w(k,r)}$ & $~~\frac{8}{85}$ & $~~\frac{2}{17}$ & $~~\frac{1}{8}$ & $~~\frac{1}{10}$\\
\hline
$~~$  $\frac{w(k,r-2)}{w(k,r)}$ & $~~\frac{5}{17}$ & $~~\frac{5}{17}$ & $~~\frac{5}{17}$ & $~~\frac{5}{17}$\\ 
\hline
$~$  $\frac{w(k-1,r-2)}{w(k,r)}$ & $~~\frac{2}{17}$ & $~~\frac{5}{34}$ & $~~\frac{5}{34}$ & $~~\frac{2}{17}$\\ 
\hline
$~$  $\frac{w(k-2,r-2)}{w(k,r)}$ & $~~\frac{1}{17}$ & $~~\frac{1}{17}$ & $~~\frac{1}{17}$ & $~~\frac{1}{17}$\\ 
\hline
\end{tabular}
\\ \medskip
~\\
Table 1: The ratios of corresponding functions to $w(k,r)$.
\end{center}

By using Table 1 and the formulas of corresponding functions, for all $k\ge3$ and $r\ge1$, we have  \medskip \\
$(a)$ $w(k,r)+1>\begin{cases}
3w(k-1,r-1)+r\ge w(k-1,r-1)+k+1,\\
2w(k-1,r)\ge8w(k-2,r-1)\ge 5w(k-2,r-1)+r\ge  w(k-1,r)+k,\\
5w(k-2,r)=17w(k-2,r-2)\ge12w(k-2,r-2)+r,\\
w(k-1,r-1)+w(k,r-1),\\
2w(k-1,r-2)+w(k,r-1)+ w(k-2,r-2)\ge w(k,r-2)\\
~~~~~~+2w(k-1,r-1)\ge w(k-1,r-2)+w(k-2,r-2)\ge3w(k-2,r-2).
\end{cases}$
\begin{center}
\begin{tabular}{| m{2.1cm} | m{1.2cm}| m{1.1cm} | m{1.2cm} | m{1.1cm} | m{1.2cm} | m{1.1cm} | m{1.2cm} | m{1.15cm} |} 
\hline
Case&  $~~~b_1$ & $~~~b_2$ & $~~~b_3$ & $~~~b_4$ & $~~~b_5$ & $~~~b_6$ & $~~~b_7$ & $~~~b_8$\\ 
\hline
$\frac{f(k-1,s-1,r)}{f(k,s,r)}$ & $~~~\frac{2}{5}$ & $~~~\frac{2}{5}$ & $~~~\frac{1}{2}$ & $~~~\frac{1}{2}$ & $~~~\frac{2}{5}$ & $\scriptsize\text{$s=1: \frac{3}{8}$}$  $\scriptsize\text{$s\ge3: \frac{2}{5}$}$ & $~~~\frac{1}{2}$ & $~~~\frac{1}{2}$\\
\hline
$\frac{f(k-2,s-2,r)}{f(k,s,r)}$ & $~~~\frac{1}{5}$  & $~~~\frac{1}{5}$ & $~~~\frac{1}{5}$ & $~~~\frac{1}{5}$ & $~~~\frac{1}{5}$ & $~~~\frac{1}{5}$ & $~~~\frac{1}{5}$ & $\scriptsize\text{$s=2:\frac{3}{16}$}$  $\scriptsize\text{$s\ge4: \frac{1}{5}$}$\\
\hline
$\frac{f(k,s+1,r-1)}{f(k,s,r)}$ & $~~~\frac{10}{21}$ & $~~~\frac{8}{17}$  & $~~~\frac{10}{21}$ & $~~~\frac{8}{17}$ & $\scriptsize\text{$r=1:\frac{1}{2}$}$  $\scriptsize\text{$r\ge3: \frac{21}{34}$}$ & $~~~\frac{5}{8}$ & $\scriptsize\text{$r=1:\frac{1}{2}$}$  $\scriptsize\text{$r\ge3: \frac{21}{34}$}$ & $\scriptsize\text{$s=0:\frac{2}{3}$}$  $\scriptsize\text{$s\ge2:\frac{5}{8}$}$\\
\hline
$\frac{f(k-1,s,r-1)}{f(k,s,r)}$ & $~~~\frac{4}{21}$ & $\scriptsize\text{$s=0:\frac{3}{17}$}$  $\scriptsize\text{$s\ge2: \frac{16}{85}$}$  & $~~~\frac{5}{21}$ & $~~~\frac{4}{17}$ & $\scriptsize\text{$r=1:\frac{1}{5}$}$  $\scriptsize\text{$r\ge3: \frac{21}{85}$}$ & $~~~\frac{1}{4}$ & $\scriptsize\text{$r=1:\frac{1}{4}$}$   $\scriptsize\text{$r\ge3: \frac{21}{68}$}$ & $\scriptsize\text{$s=0:\frac{1}{3}$}$  $\scriptsize\text{$s\ge2:\frac{5}{16}$}$\\
\hline
$\frac{f(k,s,r-1)}{f(k,s,r)}$ & $\scriptsize\text{$s=0:\frac{2}{7}$}$  $\scriptsize\text{$s\ge2:\frac{32}{105}$}$ & $\scriptsize\text{$s=0:\frac{3}{17}$}$  $\scriptsize\text{$s\ge2:\frac{16}{85}$}$ & $~~~\frac{8}{21}$ & $~~~\frac{4}{17}$ & $~~~\frac{2}{5}$ & $~~~\frac{1}{4}$ & $~~~\frac{1}{2}$  & $\scriptsize\text{$s=0:\frac{1}{3}$}$  $\scriptsize\text{$s\ge2:\frac{5}{16}$}$\\
\hline
$\frac{f(k-1,s-1,r-1)}{f(k,s,r)}$ & $~~~\frac{16}{105}$ & $~~~\frac{8}{85}$ & $\scriptsize\text{$s=1:\frac{1}{7}$}$  $\scriptsize\text{$s\ge3:\frac{16}{105}$}$ & $\scriptsize\text{$s=1:\frac{3}{34}$}$  $\scriptsize\text{$s\ge3:\frac{8}{85}$}$ & $~~~\frac{1}{5}$ & $~~~\frac{1}{8}$ & $~~~\frac{1}{5}$ & $~~~\frac{1}{8}$\\ 
\hline
$\frac{f(k-2,s-1,r-1)}{f(k,s,r)}$ & $~~~\frac{2}{21}$ & $~~~\frac{8}{85}$  & $~~~\frac{2}{21}$ & $\scriptsize\text{$s=1:\frac{3}{34}$}$  $\scriptsize\text{$s\ge3:\frac{8}{85}$}$  & $\scriptsize\text{$r=1:\frac{1}{10}$}$  $\scriptsize\text{$r\ge3:\frac{21}{170}$}$  & $~~~\frac{1}{8}$ & $\scriptsize\text{$r=1:\frac{1}{10}$}$  $\scriptsize\text{$r\ge3:\frac{21}{170}$}$  & $~~~\frac{1}{8}$\\
\hline
$\frac{f(k,s+2,r-2)}{f(k,s,r)}$ & $\scriptsize\text{$r=2:\frac{5}{21}$}$  $\scriptsize\text{$r\ge4:\frac{5}{17}$}$ & $~~~\frac{5}{17}$ &  $\scriptsize\text{$r=2:\frac{5}{21}$}$  $\scriptsize\text{$r\ge4:\frac{5}{17}$}$  & $~~~\frac{5}{17}$ & $~~~\frac{5}{17}$ & $~~~\frac{5}{17}$ & $~~~\frac{5}{17}$ &  $\scriptsize\text{$s=0:\frac{16}{51}$}$  $\scriptsize\text{$s\ge2:\frac{5}{17}$}$ \\ 
\hline
$\frac{f(k,s+1,r-2)}{f(k,s,r)}$ & $~~~\frac{4}{21}$ & $~~~\frac{2}{17}$ & $~~~\frac{5}{21}$ & $~~~\frac{5}{34}$ & $~~~\frac{16}{85}$ & $~~~\frac{2}{17}$ & $~~~\frac{4}{17}$ &  $\scriptsize\text{$s=0:\frac{8}{51}$}$  $\scriptsize\text{$s\ge2:\frac{5}{34}$}$ \\ 
\hline
$\frac{f(k-1,s+1,r-2)}{f(k,s,r)}$ & $\scriptsize\text{$r=2:\frac{2}{21}$}$  $\scriptsize\text{$r\ge4:\frac{2}{17}$}$ & $~~~\frac{2}{17}$ & $\scriptsize\text{$r=2:\frac{5}{42}$}$  $\scriptsize\text{$r\ge4:\frac{5}{34}$}$ & $~~~\frac{5}{34}$ & $~~~\frac{2}{17}$ & $~~~\frac{2}{17}$ & $~~~\frac{5}{34}$ & $\scriptsize\text{$s=0:\frac{8}{51}$}$  $\scriptsize\text{$s\ge2:\frac{5}{34}$}$ \\ 
\hline
$\frac{f(k-2,s,r-2)}{f(k,s,r)}$& $\scriptsize\text{$r=2:\frac{1}{21}$}$  $\scriptsize\text{$r\ge4:\frac{1}{17}$}$ & $~~~\frac{1}{17}$ & $\scriptsize\text{$r=2:\frac{1}{21}$}$  $\scriptsize\text{$r\ge4:\frac{1}{17}$}$ & $~~~\frac{1}{17}$ & $~~~\frac{1}{17}$ & $~~~\frac{1}{17}$ & $~~~\frac{1}{17}$ & $~~~\frac{1}{17}$\\
\hline
  $\frac{f(k-1,s,r-2)}{f(k,s,r)}$ $=\frac{f(k,s,r-2)}{f(k,s,r)}$ & $~~~\frac{2}{21}$ & $~~~\frac{1}{17}$ & $~~~\frac{2}{21}$ & $~~~\frac{1}{17}$ & $~~~\frac{8}{85}$ & $~~~\frac{1}{17}$ &  $\scriptsize\text{$s=0:\frac{3}{34}$}$  $\scriptsize\text{$s\ge2:\frac{8}{85}$}$  &  $~~~\frac{1}{17}$\\
\hline
\end{tabular}
~\\ \medskip
~\\
Table 2: The ratios of corresponding functions to $f(k,s,r)$.
\end{center}

Throughout of the proof, let $\alpha+\{\beta;\gamma\}$ denote $\alpha+\beta$ and $\alpha+\gamma$, and let $\{\alpha;\beta\}+\{\gamma;\theta\}$ denote $\alpha+\gamma$, $\alpha+\theta$, $\beta+\gamma$ and $\beta+\theta$. By using Table 2 and the formulas of corresponding functions, for all $k\ge3$ and $s+r\ge1$, we have\medskip\\
$(b)$ $f(k,s,r)+1>\begin{cases}
2f(k,s,r-1)\ge f(k,s,r-1)+s+r+1,\\
2f(k-1,s-1,r)\ge f(k-1,s-1,r)+s+r,\\
3f(k-1,s,r-1)\ge2f(k-1,s,r-1)+r,\\
8f(k-2,s-1,r-1)\ge5f(k-2,s-1,r-1)+r,\\
f(k-1,s,r-1)+f(k,s+1,r-1),\\
5f(k-2,s-2,r),\\
5f(k-1,s-1,r-1),
\end{cases}$\medskip \\
and for all $k\ge3$ and $r\ge2$, we also have\\
$(c)$ $f(k,s,r)+1>\begin{cases}
f(k-1,s+1,r-2)+2f(k-1,s,r-1)+r,\\
7f(k-1,s,r-2)+r,\\
17f(k-2,s,r-2)\ge14f(k-2,s,r-2)+r,\\
2f(k-1,s,r-1)+\{f(k,s+1,r-2); f(k,s+2,r-2)\},\\
4f(k-1,s,r-2)+f(k,s+1,r-1),\\
\{f(k,s+1,r-2)+f(k,s,r-1); f(k-1,s+1,r-2)+f(k,s+1,r-1)\}\\
~~~~~~~~~~~~~~~~~~~~~~~+\{3f(k-2,s,r-2); 2f(k-1,s,r-2)\},\\
3f(k,s+1,r-2)+f(k-2,s,r-2)+max\{f(k-1,s,r-2), 4\},\\
3f(k-1,s+1,r-2)+\{3f(k-1,s+1,r-2)+f(k-2,s,r-2);\\ 
~~~~~~~~~~~~~~~~~6f(k-2,s,r-2); 11f(k-2,s,r-2)-f(k-1,s+1,r-2)\},\\
3f(k-1,s+1,r-2)+f(k,s+1,r-2)+\{4f(k-2,s,r-2); f(k-1,\\
~~~~~~s+1, k-2)+2f(k-2,s,r-2); 2f(k-1,s,r-2)+f(k-2,s,r-2)\}.
\end{cases}$

\section{Proof of Theorems \ref{t1.3} and \ref{t1.4}}
By constructions in Section 2, it suffices to show that $GR_k((k-r)K_3,rH)\le w(k,r)+1$ with $H\in\{H_1,H_2\}$ and $GR_k((k-s-r)P_3,sK_3,rH_3)\le f(k,s,r)+1$ for all $k\ge1$ and $s,r$ with $0\le s\le k$ and $0\le r\le k$. For simplicity of notations, we use C1.4 and C1.5 to represent the cases $GR_k((k-r)K_3,rH)$ with $H\in\{H_1,H_2\}$ and $GR_k((k-s-r)P_3,sK_3,rH_3)$, respectively. We proceed the proof by induction on $k+r$ for C1.4 and $k+2r$ for C1.5. The case for $k=1$ is trivial. By Theorems \ref{MC}-\ref{SPR}, $GR_2(H)=R_2(H)=w(2,2)+1$, $GR_2(K_3,H)=R_2(K_3,H)=w(2,1)+1$ and $GR_2(K_3)=R_2(K_3)=w(2,0)+1$ for C1.4; $GR_2(H_3)=R_2(H_3)=f(2,0,2)+1$, $GR_2(K_3)=R_2(K_3)=f(2,2,0)+1$, $GR_2(K_3,H_3)=R(K_3,H_3)=f(2,1,1)+1$, $GR_2(P_3,H_3)=R(P_3,H_3)=f(2,0,1)+1$ and $GR_2(P_3,K_3)=R(P_3,K_3)=f(2,1,0)+1$ for C1.5. The case $r=0$ for C1.4 is Theorem \ref{c}, and the cases $s+r=0$ and $s=k$ for C1.5 are Theorems \ref{4} and \ref{c}, respectively. Therefore, we may assume that $k\ge3$ and $1\le r\le k$ for C1.4; $k\ge3$, $1\le s+r\le k$ and $0\le s<k$ for C1.5. Suppose that Theorem \ref{t1.3} holds for all $k^{'}+r^{'}<k+r$ and Theorem \ref{t1.4} holds for all $k^{'}+2r^{'}<k+2r$. Set $G=K_{w(k,r)+1}$ for C1.4 and $G=K_{f(k,s,r)+1}$ for C1.5. Let $c: E(G)\xrightarrow{} [k]$ be any Gallai $k$-coloring of $G$. Suppose $(G,c)$ contains neither monochromatic copy of $H\in\{H_1,H_2\}$ in any of the first $r$ colors nor monochromatic copy of $K_3$ in any of the last $k-r$ colors for C1.4 and $(G,c)$ contains no monochromatic copy of $H_3$ in any of the first $r$ colors, monochromatic copy of $K_3$ in any of the middle $s$ colors and monochromatic copy of $P_3$ in any of the last $k-s-r$ colors for C1.5. Choose $(G,c)$ with $k$ minimum.\medskip

Let $X_1$ and $X_2$ be disjoint sets of $V(G)$ such that $X_1$ is color $i$-adjacent to $X_2$, $i\in [k]$. It is easy to check that $G[X_1\cup X_2]$ contains a monochromatic copy of $H\in\{H_1,H_2\}$ in color $i$ if condition 1 holds or $H\in\{H_1,H_2,H_3\}$ in color $i$ if any one of conditions 2-4 holds.
\begin{enumerate}
\item Both $G[X_1]$ and $G[X_2]$ have an edge in color $i$ and $|X_1\cup X_2|\ge5$.
\item $|X_1|\ge2$ and $G[X_2]$ contains a $K_3$ in color $i$.
\item $|X_1|\ge1$ and $G[X_2]$ contains a $K_4-e$ in color $i$.
\item $G[X_1]$ has an edge in color $i$ and $G[X_2]$ has a $P_3$ in color $i$.
\end{enumerate}

Let $t_1,t_2,\ldots,t_m\in V(G)$ be a maximum sequence of vertices chosen as follows: for each $j\in [m]$, all edges between $t_j$ and $V(G)\setminus\{t_1,t_2,\ldots,t_j\}$ are colored the same color under $c$. Let $T=\{t_1,t_2,\ldots, t_m\}$. Notice that $T$ is possibly empty. For each $t_j\in T$, let $c(t_j)$ be the unique color on the edges between $t_j$ and $V(G)\setminus\{t_1,t_2,\ldots,t_j\}$.
\begin{claim}\label{c1}
$c(t_i)\neq c(t_j)$ for all $i,j\in [m]$ with $i\neq j$. Thus all colors in $\{c(t_1),\ldots,c(t_m)\}$ are assigned to $H\in\{H_1,H_2,H_3\}$.
\end{claim}
\begin{proof} 
Suppose that $c(t_i)=c(t_j)$ for some $i,j\in [m]$ with $i\neq j$. We may assume that $t_j$ is the first vertex in the sequence $t_1,\ldots,t_m$ such that $c(t_i)=c(t_j)$ for some $i\in [m]$ with $i<j$. We may further assume that the color $c(t_i)$ is red. Thus the edge $t_it_j$ is colored with red under $c$. Let $A=V(G)\setminus\{t_1,t_2,\ldots,t_j\}$. Then all the edges between $\{t_i,t_j\}$ and $A$ are colored with red under $c$. We start with the arguments on C1.4. By the pigeonhole principle, $j\le k+1$. Note that $|A|\ge|G|-(k+1)\ge3$ for all $k\ge3$ and $1\le r\le k$. Since $t_it_j$ is a red edge, red can not be the color assigned to $K_3$. By condition 1, there is no red edge in $(G[A],c)$. By induction, $|A|\le w(k-1,r-1)$. Then by $(a)$, $|G|\le w(k-1,r-1)+k+1<w(k,r)+1$, contrary to the fact that $|G|=w(k,r)+1$. We now turn to the discussion on C1.5. By the property of $T$, we can see that $c(t_{j-1})$ and $c(t_j)$ are the only two colors in $\{c(t_1),\ldots,c(t_j)\}$ that could be assigned to $P_3$. Since $|G|-s-r\ge3$ for all $k\ge3$, $1\le s+r\le k$ and $0\le s<k$, no color in $\{c(t_1),\ldots,c(t_j)\}$ can be assigned to $P_3$. By the pigeonhole principle, $j\le s+r+1$. Note that $|A|\ge|G|-(s+r+1)\ge3$ for all $k\ge3$, $1\le s+r\le k$ and $0\le s<k$. Since $t_it_j$ is a red edge, red must be the color assigned to $H_3$. By condition 4, there is no red $P_3$ in $(G[A],c)$. By induction, $|A|\le f(k,s,r-1)$. Then by $(b)$, $|G|\le f(k,s,r-1)+s+r+1<f(k,s,r)+1$, which is a contradiction. Thus $c(t_i)\neq c(t_j)$ for all $i,j\in [m]$ with $i\neq j$. \medskip

Similar to the argument above, no color in $\{c(t_1),\ldots,c(t_m)\}$ can be assigned to $P_3$. Thus $|T|\le k$ for C1.4 and $|T|\le s+r$ for C1.5. Suppose that there is a color in $\{c(t_1),\ldots,c(t_m)\}$ which is assigned to $K_3$, say green. Then $G\setminus T$ contains no green edge. By induction, $|G\setminus T|\le w(k-1,r)$ and $|G\setminus T|\le f(k-1,s-1,r)$. By $(a)$ and $(b)$, $|G|\le w(k-1,r)+k<w(k,r)+1$ and $|G|\le f(k-1,s-1,r)+s+r<f(k,s,r)+1$, which are impossible. Thus all colors in $\{c(t_1),\ldots,c(t_m)\}$ have to be assigned to $H\in\{H_1,H_2,H_3\}$.
\end{proof}

By Claim \ref{c1}, we see that $|T|\le r$. Consider a Gallai-partition of $G\setminus T$ with parts $V_1,V_2,\ldots,V_{\ell}$ such that $\ell\ge2$ is as small as possible. Assume that $|V_1|\ge|V_2|\ge \ldots \ge |V_{\ell}|$. Let $\mathcal{G}$ be the reduced graph of $G\setminus T$ with vertices $v_1,\ldots,v_{\ell}$. By Theorem \ref{Gallai}, we may further assume that the edges of $\mathcal{G}$ are colored with red or blue. It is obvious that any monochromatic copy of $H\in\{H_1,H_2,H_3\}$, $K_3$ and $P_3$ in $\mathcal{G}$ would yield a monochromatic copy of  $H\in\{H_1,H_2,H_3\}$, $K_3$ and $P_3$ in $G\setminus T$, respectively. Let
\begin{center}
$\mathcal{V}_r=\{V_i$ $|$ $V_i$ is red-adjacent to $V_1$ under $c$, $i\in\{2,\ldots,\ell\}\}$ and \\
$\mathcal{V}_b=\{V_i$ $|$ $V_i$ is blue-adjacent to $V_1$ under $c$, $i\in\{2,\ldots,\ell\}\}$.\ \ \ \ \ \
\end{center}
Let $R=\bigcup_{V_i\in\mathcal{V}_r}V_i$ and $B=\bigcup_{V_i\in\mathcal{V}_b}V_i$. Then $|G|=|V_1|+|R|+|B|+|T|=|V_1\cup T|+|R|+|B|=|V_1\cup B|+|R|+|T|=|V_1\cup B\cup T|+|R|$, denoted by $(*)$. Without loss of generality, we may assume that $|B|\leq |R|$. It is easily seen that $|R|\ge2$, for otherwise the vertex in $R$ or $B$ can be added to $T$, contrary to the maximality of $m$ in $T$. Thus red can not be assigned to $P_3$ for C1.5.
\begin{claim}\label{c2}
$|V_1|\ge2$.
\end{claim}
\begin{proof}
Suppose that $|V_1|=1$. Then $(G\setminus T,c)$ is only colored with red and blue. We may first assume that red is the color assigned to $K_3$ while blue is the color assigned to $P_3$. By Theorem \ref{SPR}, $\ell\le R(P_3,K_3)-1=4$. By Claim \ref{c1} and $k\ge3$, we see that $r\ge1$. But then $|G\setminus T|\ge f(3,1,1)>4$, a contradiction. We next assume that red is the color assigned to $H_3$ while blue is the color assigned to $P_3$. By Theorem \ref{SPR}, $\ell\le R(P_3,H_3)-1=6$. By Claim \ref{c1} and $k\ge3$, we have $r\ge2$. Then $|G\setminus T|\ge f(3,0,2)-1>6$, which is also a contradiction. Now we give the proofs of C1.4 and C1.5 together. Assume that red and blue are the colors assigned to $K_3$. By Theorem \ref{GG}, $\ell\le R_2(K_3)-1=5$. By Claim \ref{c1} and $k\ge3$, we have $r\ge1$. Then $|G\setminus T|\ge f(3,2,1)=w(3,1)>5$, a contradiction. We next assume that red is the color assigned to $H\in\{H_1,H_2,H_3\}$ while blue is the color assigned to $K_3$. By Theorem \ref{MC}, $\ell\le R(K_3,H)-1=8$ for $H\in\{H_1,H_2\}$ and $\ell\le R(K_3,H_3)-1=10$. By Claim \ref{c1} and $k\ge3$, we have $r\ge2$. Then  $|G\setminus T|\ge f(3,1,2)-1>w(3,2)-1>\ell$, a contradiction. Finally, we assume that red and blue are the colors assigned to $H\in\{H_1,H_2,H_3\}$. By Theorem \ref{R2}, $\ell\le R_2(H)-1=17$ for $H\in\{H_1,H_2\}$ and $\ell\le R_2(H_3)-1=21$. By Claim \ref{c1} and $k\ge3$, we have $r\ge3$. Then $|G\setminus T|\ge f(3,0,3)-2=w(3,3)-2>\ell$, which is impossible. Hence, $|V_1|\ge2$.
\end{proof}
\begin{claim}\label{c3}
$|B|\ge1$ when red is the color assigned to $K_3$ and $|B|\ge2$ when red is the color assigned to $H\in\{H_1,H_2,H_3\}$.
\end{claim}
\begin{proof}
Let red be the color assigned to $K_3$. Suppose $|B|=0$. Recall that $|R|\ge2$. By Claims \ref{c1} and \ref{c2}, there is no red edge in either $(G[V_1\cup T],c)$ or $(G[R],c)$. By induction, $|V_1\cup T|\le w(k-1,r)$, $|R|\le w(k-1,r)$ for C1.4 and $|V_1\cup T|\le f(k-1,s-1,r)$, $|R|\le f(k-1,s-1,r)$ for C1.5. By $(*)$, $(a)$ and $(b)$, we have $|G|\le 2w(k-1,r)<w(k,r)+1$ and $|G|\le 2f(k-1,s-1,r)<f(k,s,r)+1$, contrary to the facts that $|G|=w(k,r)+1$ and $|G|=f(k,s,r)+1$.\medskip

Let red be the color assigned to $H\in\{H_1,H_2,H_3\}$. Suppose that $|B|\le 1$. We may first assume that no vertex in $T$ is red-adjacent to $V(G)\setminus T$ under $c$. If $(G[V_1],c)$ contains no red edge, then there is no red edge in $(G[V_1\cup B\cup T],c)$. Recall that $|V_1|\ge2$. So by condition 2, there can not be red $K_3$ in $(G[R],c)$. By induction, $|V_1\cup B\cup T|\le w(k-1,r-1)$, $|R|\le w(k,r-1)$ for C1.4 and $|V_1\cup B\cup T|\le f(k-1,s,r-1)$, $|R|\le f(k,s+1,r-1)$ for C1.5. By $(*)$, $(a)$ and $(b)$, we have $|G|\le w(k-1,r-1)+w(k,r-1)<w(k,r)+1$ and $|G|\le f(k-1,s,r-1)+f(k,s+1,r-1)<f(k,s,r)+1$, which are impossible. Thus $(G[V_1],c)$ contains red edges. Note that $|V_1\cup R|\ge5$ for C1.4 as $|G|>5+r$. By conditions 1 and 4, there is no red edge in $(G[R],c)$ and no red $K_3$ in $(G[V_1],c)$ for C1.4. Similar to the argument above, $|G|<w(k,r)+1$. As for C1.5, by condition 4, we see that $(G[R],c)$ has no red $P_3$. Similar to above, we only need to consider the case that  $(G[R],c)$ contains red edges. By condition 4, there is no red $P_3$ in $(G[V_1],c)$ and neither is $(G[V_1\cup B\cup T],c)$. By induction, $|V_1\cup B\cup T|\le f(k,s,r-1)$ and $|R|\le f(k,s,r-1)$. By $(*)$ and $(b)$, we have $|G|\le 2f(k,s,r-1)<f(k,s,r)+1$, a contradiction.\medskip 

We next assume that $T$ has a vertex which is red-adjacent to $V(G)\setminus T$ under $c$. By condition 3, there can not be red edge in either $(G[V_1],c)$ or $(G[R],c)$ and so $(G[V_1\cup B],c)$ contains no red edge. By induction, $|V_1\cup B|\le w(k-1,r-1)$, $|R|\le w(k-1,r-1)$ for C1.4 and $|V_1\cup B|\le f(k-1,s,r-1)$, $|R|\le f(k-1,s,r-1)$ for C1.5. By $(*)$, $(a)$ and $(b)$, we have $|G|\le 2w(k-1,r-1)+r<w(k,r)+1$ and $|G|\le 2f(k-1,s,r-1)+r<f(k,s,r)+1$, which are impossible.
\end{proof}
By Claims \ref{c2} and \ref{c3}, we see that blue is not the color assigned to $P_3$ for C1.5. Thus we only need to consider the following three cases.
\begin{case}
Red and blue are the colors assigned to $K_3$.	
\end{case}
By Claims \ref{c1}-\ref{c3}, $(G[V_1\cup T],c)$ contains neither red nor blue edge. By induction, $|V_1\cup T|\le w(k-2,r)$ for C1.4 and $|V_1\cup T|\le f(k-2,s-2,r)$ for C1.5. By Theorem \ref{GG}, $\ell \le R_2(K_3)-1=5$. Then by $(*)$, $(a)$ and $(b)$, we have $|G|\le 5w(k-2,r)<w(k,s)+1$ and $|G|\le 5f(k-2,s-2,r)<f(k,s,r)+1$, contrary to the facts that $|G|=w(k,r)+1$ and $|G|=f(k,s,r)+1$.
\begin{case}
Red is the color assigned to $H\in\{H_1,H_2,H_3\}$ while blue is the color assigned to $K_3$.	
\end{case}
We claim that no vertex in $T$ is red-adjacent to $V(G)\setminus T$ under $c$. Suppose not. Then by condition 3, there is no red edge in either $(G[V_1],c)$ or $(G[R],c)$. By Claim \ref{c3}, there is no blue edge in $(G[V_1],c)$ and thus $(G[V_1],c)$ contains neither red nor blue edge. By induction, $|V_1|\le w(k-2,r-1)$ for C1.4 and $|V_1|\le f(k-2,s-1,r-1)$ for C1.5. In order to avoid blue $K_3$, there has to be at most two parts of $\{V_2,\ldots,V_{\ell}\}$ in $R$, which implies that $|B|\le |R|\le 2|V_1|$. By $(*)$, $(a)$ and $(b)$, we have $|G|\le 5w(k-2,r-1)+r<w(k,r)+1$ and $|G|\le 5f(k-2,s-1,r-1)+r<f(k,s,r)+1$, contrary to the facts that $|G|=w(k,r)+1$ and $|G|=f(k,s,r)+1$.\medskip

By Claim \ref{c3}, we see that no vertex in $T$ is red or blue-adjacent to $V(G)\setminus T$ under $c$. Furthermore, $(G[V_1\cup T],c)$ contains no blue edge. If $(G[V_1],c)$ contains no red edge, then $(G[V_1\cup T],c)$ contains neither red nor blue edge. By induction, $|V_1\cup T|\le w(k-2,r-1)$ and $|V_1\cup T|\le f(k-2,s-1,r-1)$. Note that $\ell\le R(K_3,H)-1=8$ for C1.4. By $(*)$ and $(a)$, we have $|G|\le 8w(k-2,r-1)<w(k,r)+1$, which is impossible. As for C1.5, we know that $\ell\le R(K_3,H_3)-1=10$. If $\ell\le8$, then $|G|\le8|V_1\cup T|$. If $9\le\ell\le10$, then there is a red $K_4$ in the reduced graph $\mathcal{G}$ as $R(K_3,K_4)=9$. We know that a monochromatic $K_4$ in $\mathcal{G}$ has to be four vertices in $G\setminus T$. Thus $|G|\le 6|V_1\cup T|+4$. By $(*)$ and $(b)$, we have $|G|\le 6f(k-2,s-1,r-1)+4\le8f(k-2,s-1,r-1)<f(k,s,r)+1$, which is a contradiction. Thus there exist red edges in $(G[V_1],c)$.\medskip

We first consider the proof of C1.4. Note that $|V_1\cup R|\ge5$ as $|G|>6+r$. By conditions 1 and 2, there is no red edge in $(G[R],c)$ and no red $K_3$ in $(G[V_1\cup T],c)$. By induction, $|B|\le|R|\le w(k-1,r-1)$ and $|V_1\cup T|\le w(k-1,r-1)$. By $(*)$ and $(a)$, we have $|G|\le 3w(k-1,r-1)<w(k,r)+1$, a contradiction. We now consider the proof of C1.5. By condition 4, there is no red $P_3$ in $(G[R],c)$. If $(G[R],c)$ contains no red edge, then by condition 2, there is no red $K_3$ in $(G[V_1\cup T],c)$. By induction, $|V_1\cup T|\le f(k-1,s,r-1)$ and $|B|\le|R|\le f(k-1,s,r-1)$. So by $(*)$ and $(b)$, we have $|G|\le 3f(k-1,s,r-1)<f(k,s,r)+1$, a contradiction. Thus $(G[R],c)$ contains red edges. By condition 4, there is no red $P_3$ in either $(G[V_1\cup T],c)$ or $(G[R],c)$. By induction, $|V_1\cup T|\le f(k-1,s-1,r-1)$. As $(G[R],c)$ contains no red $P_3$, we have the following claim.
\begin{claim}\label{c4}
$|B|\le|R|\le 2|V_1|$.
\end{claim} 
\begin{proof}
Note that $(G[R],c)$ contains no blue $K_3$. By Theorem \ref{SPR}, $R$ has at most $R(P_3,K_3)-1=4$ parts of $\{V_2,\ldots,V_{\ell}\}$ and thus there are at most two independent red edges between the parts in $R$. It's easy to check that $|R|\le4$ when $R$ has two such edges, $|R|\le |V_1|+2$ when $R$ has only one such edge and $|R|\le 2|V_1|$ when $R$ does not contain such an edge. By Claim \ref{c2}, we know that $|V_1|\ge2$. Recall that $|B|\le|R|$. Thus $|B|\le|R|\le 2|V_1|$.
\end{proof}
\noindent
By Claim \ref{c4}, $(*)$ and $(b)$, we have $|G|\le 5f(k-1,s-1,r-1)<f(k,s,r)+1$, which is impossible.
\begin{case}
Red and blue are the colors assigned to $H\in\{H_1,H_2,H_3\}$.	
\end{case}
Define $Y_1=\{V_i:|V_i|=1, i\in\{2,\ldots,\ell\}\}$ and $Y_2=\{V_i: |V_i|\ge2, i\in\{2,\ldots,\ell\}\}$. Then $|Y_1\cup Y_2|=|R\cup B|$. Let $|R\cap Y_t|$ and $|B\cap Y_t|$ be the number of the common parts in $\{V_2,\ldots,V_{\ell}\}$, where $t=1,2$. Suppose $(G[R],c)$ contains no red edge ($res$. red $P_3$) or $(G[B],c)$ contains no blue edge ($res$. blue $P_3$). Clearly, $|R\cap Y_2|\le3$ or $|B\cap Y_2|\le3$, for otherwise there is a blue $H\in\{H_1,H_2,H_3\}$ in $(G[R],c)$ or a red $H\in\{H_1,H_2,H_3\}$ in $(G[B],c)$. In particular, we have the following four facts:
\begin{enumerate}
\item[(1)] If $|R\cap Y_2|=3$ or $|B\cap Y_2|=3$, then $|R\cap Y_1|=0$ or $|B\cap Y_1|=0$,
\item[(2)] If $|R\cap Y_2|=2$ or $|B\cap Y_2|=2$, then $|R\cap Y_1|\le1$ or $|B\cap Y_1|\le1$ ($res$. $|R\cap Y_1|\le2$ or $|B\cap Y_1|\le2$),
\item[(3)] If $|R\cap Y_2|=1$ or $|B\cap Y_2|=1$, then $|R\cap Y_1|\le2$ or $|B\cap Y_1|\le2$ ($res$. $|R\cap Y_1|\le4$ or $|B\cap Y_1|\le4$),
\item[(4)] If $|R\cap Y_2|=0$ or $|B\cap Y_2|=0$, then $|R\cap Y_1|\le4$ or $|B\cap Y_1|\le4$ ($res$. $|R\cap Y_1|\le6$ or $|B\cap Y_1|\le6$),
\item[(5)] $|R|\le 3|V_1|$ or $|B|\le 3|V_1|$.
\end{enumerate}
\begin{proof}
We only consider the proof for $R$. The proof for $B$ is similar. We first consider the case that $(G[R],c)$ contains no red edge. Obviously, all the parts in $R$ are blue-adjacent to each other. In order to avoid a blue $H\in\{H_1,H_2,H_3\}$ in $(G[R],c)$, $R$ has at most three parts unless $|R\cap Y_2|=0$. However, $|R\cap Y_1|\le4$ if $|R\cap Y_2|=0$, for otherwise we obtain a blue $K_5$ that contains all blue $H\in\{H_1,H_2,H_3\}$.\medskip

Now we consider the case that $(G[R],c)$ contains no red $P_3$. We can see that $Y_1$ is blue-adjacent to $Y_2$ in $(G[R],c)$. Also, all the parts of $Y_2$ are blue-adjacent to each other. For fact (1), there is a blue $K_4-e$ in $(G[R\cap Y_2],c)$. By condition 3, if $|R\cap Y_1|\ge1$, then we can obtain a blue $H\in\{H_1,H_2,H_3\}$, which is impossible. Suppose that $|R\cap Y_1|\ge3$ for fact (2). Let $V_i,V_j,V_k\in Y_1$. In order to avoid a red $P_3$, there has to be a blue $P_3$ in $(G[V_i\cup V_j\cup V_k],c)$. Note that $(G[R\cap Y_2],c)$ contains blue edges. By condition 4, we obtain a blue $H\in\{H_1,H_2,H_3\}$, a contradiction.  Suppose that $|R\cap Y_1|\ge5$ for fact (3). Since $R(P_3,K_3)=5$, there is a blue $K_3$ in $(G[R\cap Y_1],c)$. By condition 2, we get a blue $H\in\{H_1,H_2,H_3\}$ as $|Y_1|\ge2$, which is also a contradiction. For fact (4), we see that there are at most $R(P_3,H_3)-1=6$ parts in $(G[R],c)$, for otherwise we can obtain a blue $H_3$ that also contains all blue $H\in\{H_1,H_2\}$. Recall that $|V_1|\ge2$. By facts (1)-(4), the fact (5) holds.
\end{proof}
\begin{claim}\label{c5}
No vertex in $T$ is red or blue-adjacent to $V(G)\setminus T$ under $c$. 
\end{claim}
\begin{proof}
Suppose that $T$ has a vertex which is red-adjacent to $V(G)\setminus T$ under $c$. By condition 3, there can not be red edge in either $(G[V_1],c)$ or $(G[R],c)$. By Claim \ref{c3} and condition 2, there is no blue $K_3$ in $(G[V_1],c)$. So by induction, $|V_1|\le w(k-1,r-1)$, $|B|\le|R|\le w(k-1,r-1)$ for C1.4 and $|V_1|\le f(k-1,s+1,r-2)$, $|B|\le |R|\le f(k-1,s,r-1)$ for C1.5. By $(*)$, $(a)$ and $(c)$, we have $|G|\le 3w(k-1,r-1)+r<w(k,r)+1$ and $|G|\le f(k-1,s+1,r-2)+ 2f(k-1,s,r-1)+r<f(k,s,r)+1$, contrary to the facts that $|G|=w(k,r)+1$ and $|G|=f(k,s,r)+1$.\medskip

Suppose that $T$ has a vertex which is blue-adjacent to $V(G)\setminus T$ under $c$. By Claim \ref{c3}, we know that $|B|\ge2$. By condition 3 again, there can not be blue edge in either $(G[V_1],c)$ or $(G[B],c)$. Assume that $(G[V_1],c)$ contains red edges. Note that $|V_1\cup R|\ge5$ for C1.4 as $|G|>6+r$. By condition 1, there is no red edge in $(G[R],c)$ for C1.4. By induction, $|V_1|\le w(k-1,r-1)$ and $|B|\le|R|\le w(k-1,r-1)$.  Similar to above, $|G|<w(k,s)+1$. As for C1.5, if $(G[R],c)$ contains no red edge, then by condition 2,  there is no red $K_3$ in $(G[V_1],c)$. By induction, $|V_1|\le f(k-1,s+1,r-2)$ and $|B|\le|R|\le f(k-1,s,r-1)$. Similar to above, $|G|<f(k,s,r)+1$. Thus $(G[R],c)$ contains red edges. By condition 4, there is no red $P_3$ in either $(G[V_1],c)$ or $(G[R],c)$. Thus by induction, $|V_1|\le f(k-1,s,r-2)$. By fact (5), $|B|\le|R|\le 3|V_1|$.  Then by $(*)$ and $(c)$, we have $|G|\le 7f(k-1,s,r-2)+r<f(k,s,r)+1$, which is a contradiction. \medskip

Assume that $(G[V_1],c)$ contains no red edge. Then there is neither red nor blue edge in $(G[V_1],c)$. By induction, $|V_1|\le w(k-2,r-2)$ and $|V_1|\le f(k-2,s,r-2)$. By condition 2, there is no red $K_3$ in $(G[R],c)$ and thus $R$ has at most $R(K_3,H)-1=8$  and $R(K_3,H_3)-1=10$ parts of $\{V_2,\ldots,V_{\ell}\}$ for C1.4 and  C1.5, respectively. Since $(G[B],c)$ has no blue edge, $|B|\le 3|V_1|$ by fact (5). Then by $(*)$, $(a)$ and $(c)$,  we have $|G|\le 12w(k-2,r-2)+r<w(k,r)+1$ and $|G|\le$ $14f(k-2,s,r-2)+r<f(k,s,r)+1$, which are impossible.
\end{proof}

Clearly, if $(G[V_1],c)$ contains red or blue edges, then so does $(G[V_1\cup T],c)$. By Claim \ref{c5}, if $(G[V_1],c)$ does not contain red or blue edge, then neither does $(G[V_1\cup T],c)$. We next consider the following two subcases.
\begin{subcase}
$(G[V_1\cup T],c)$ contains no blue edge.
\end{subcase}
Suppose that there is no red edge in $(G[V_1\cup T],c)$. Then by induction, $|V_1\cup T|\le w(k-2,r-2)$ and $|V_1\cup T|\le f(k-2,s,r-2)$. Clearly, $\ell\le R_2(H)-1=17$ for C1.4. As for C1.5, by condition 2, there is no red $K_3$ in $(G[R],c)$. So there are at most $R(K_3,H_3)-1=10$ parts of $\{V_2,\ldots,v_{\ell}\}$ in $R$. Note that the monochromatic $K_4$ in the reduced graph $\mathcal{G}$ has to be four vertices in $G\setminus T$. Therefore, $|B|\le|R|\le 6|V_1|+4\le8|V_1|$ as $R(K_3,K_4)=9$. Then by $(*)$, $(a)$ and $(c)$, we have $|G|\le 17w(k-2,r-2)<w(k,r)+1$ and $|G|\le 17f(k-2,s,r-2)<f(k,s,r)+1$, contrary to the facts that $|G|=w(k,r)+1$ and $|G|=f(k,s,r)+1$. Thus $(G[V_1\cup T],c)$ contains red edges.\medskip

Note that $|V_1\cup R|\ge5$ for C1.4 as $|G|>6+r$. By conditions 1 and 4, there is no red edge and red $P_3$ in $(G[R],c)$ for C1.4 and C1.5, respectively. If $(G[R],c)$ contains no red edge, then by condition 2, there is no red $K_3$ in $(G[V_1\cup T],c)$. By induction, $|V_1\cup T|\le w(k-1,r-2)\le w(k,r-2)$, $|B|\le|R|\le w(k-1,r-1)$ for C1.4 and $|V_1\cup T|\le f(k-1,s+1,r-2)$, $|B|\le|R|\le f(k-1,s,r-1)$ for C1.5. Then by $(*)$, $(a)$ and $(b)$, we have $|G|\le w(k,r-2)+2w(k-1,r-1)<w(k,r)+1$ and $|G|\le f(k-1,s+1,r-2)+2f(k-1,s,r-1)<f(k,s,r)+1$, which are impossible. Thus $(G[R],c)$ contains red edges for C1.5. By condition 4, there is no red $P_3$ in either $(G[V_1\cup T],c)$ or $(G[R],c)$. By induction, $|V_1\cup T|\le f(k-1,s,r-2)$. By fact (5), we know that $|B|\le |R|\le3|V_1|$. Then by $(*)$ and $(b)$, we have $|G|\le 7f(k-1,s,r-2)<f(k,s,r)+1$, which is a contradiction.
\begin{subcase}\label{sub 3.2}
$(G[V_1\cup T],c)$ contains blue edges.	
\end{subcase}
Before starting the proof of Subcase \ref{sub 3.2}, we first study many related properties of $(G[B],c)$. Suppose that $(G[B],c)$ contains no blue edge. Then we have the following facts.
\begin{enumerate}
\item[(6)] If $|B\cap Y_2|=3$, then $|B|\le3w(k-2,r-2)$ for C1.4 and $|B|\le3f(k-2,s,r-2)$ for C1.5,
\item[(7)] If $|B\cap Y_2|\le2$, then $|B|\le |V_1|+w(k-2,r-2)$ for C1.4 and $|B|\le 2f(k-1,s,r-2)$ or $|B|\le |V_1|+f(k-2,s,r-2)$ for C1.5.
\end{enumerate}
\begin{proof}
Let $|B\cap Y_2|=3$. By condition 3, there can not be red edge in each part of $B\cap Y_2$. By induction and fact (1), the fact (6) holds. Let $|B\cap Y_2|\le1$. By facts (3) and (4), we have $|B|\le |V_1|+2$. Let $|B\cap Y_2|=2$. By fact (2), we know that $|B\cap Y_1|\le1$. Assume that $|B\cap Y_1|=1$. By the same reason above, we have $|B|\le 2w(k-2,r-2)+1$ for C1.4 and $|B|\le 2f(k-2,s,r-2)+1$ for C1.5. We next assume that $|B\cap Y_1|=0$. Now $B$ only has two parts. By condition 1, we can see that there is at least one part without red edge for C1.4 unless $|B|=4$. By induction, $|B|\le |V_1|+w(k-2,r-2)$. For C1.5, if both two parts have red edges, then there can not be red $P_3$ in either of them by condition 4. So by induction, $|B|\le 2f(k-1,s,r-2)$. If there is at least one part without red edge, then by induction, $|B|\le |V_1|+f(k-2,s,r-2)$. 
\end{proof}

\textbf{Proof of C1.4.} Suppose that $(G[V_1\cup T],c)$ contains red edges.  Then $(G[V_1\cup T],c)$ contains red and blue edges and so does $(G[V_1],c)$ by Claim \ref{c5}.  This means that $|V_1|\ge3$. By Claim \ref{c3}, we have $|B|\ge2$. Recall that $|R|\ge2$. So by conditions 1 and 2, there is no red edge in $(G[R],c)$ and neither red nor blue $K_3$ in $(G[V_1\cup T],c)$. By induction, $|V_1\cup T|\le w(k,r-2)$ and $|B|\le |R|\le w(k-1,r-1)$. By $(a)$, $|G|=|V_1\cup T|+|R|+|B|\le w(k,r-2)+2w(k-1,r-1)<w(k,r)+1$, a contradiction. Thus there is no red egde in $(G[V_1\cup T],c)$. Similar to above, there is no blue $K_3$ in $(G[V_1\cup T],c)$ and no red $K_3$ in $(G[R],c)$. By induction, $|V_1\cup T|\le w(k-1,r-2)$ and $|R|\le w(k,r-1)$. Note that $|V_1\cup B|\ge5$ as $|G|>w(k,r-1)+r+4$ and $r\ge2$. So there is no blue edge in $(G[B],c)$. Thus by $(a)$ and facts (6) and (7), we can get that $|B|\le 3(k-2,r-2)\le w(k-1,r-2)+w(k-2,r-2)$. Then by $(*)$ and $(a)$, we have $|G|\le 2w(k-1,r-2)+w(k,r-1)+ w(k-2,r-2)<w(k,r)+1$, contrary to the fact that $|G|=w(k,r)+1$.\medskip

\textbf{Proof of C1.5.} By condition 4, we see that there can not be blue $P_3$ in $(G[B],c)$. We then consider following two subcases based on Subcase \ref{sub 3.2}.
\begin{Subcase}
$(G[B],c)$ contains blue edges. 
\end{Subcase}
By condition 4, there is no blue $P_3$ in $(G[V_1\cup T],c)$. Suppose that $(G[V_1\cup T],c)$ contains red edges. If $(G[R],c)$ has red edges, then by condition 4, there is no red $P_3$ in either $(G[V_1\cup T],c)$ or $(G[R],c)$. By induction, $|V_1\cup T|\le f(k,s,r-2)$. By fact (5), we have $|B|\le|R|\le3|V_1|$. Then by $(*)$ and $(b)$, we have $|G|\le 7f(k,s,r-2)=7f(k-1,s,r-2)<f(k,s,r)+1$, a contradiction. If $(G[R],c)$ contains no red edge, then by condition 2, there is no red $K_3$ in $(G[V_1\cup T],c)$. By induction, $|V_1\cup T|\le f(k,s+1,r-2)$ and $|B|\le|R|\le f(k-1,s,r-1)$. Then by $(*)$ and $(c)$, we have $|G|\le f(k,s+1,r-2)+2f(k-1,s,r-1)<f(k,s,r)+1$, which is a contradiction. So there can not be red edge in $(G[V_1\cup T],c)$. By condition 2, there is no red $K_3$ in $(G[R],c)$. By induction, $|V_1\cup T|\le f(k-1,s,r-2)$ and  $|R|\le f(k,s+1,r-1)$. Since there is no blue $P_3$ in $(G[B],c)$, we have $|B|\le3|V_1|$ by fact (5).  Then by $(*)$ and $(c)$, we have $|G|\le 4f(k-1,s,r-2)+f(k,s+1,r-1)<f(k,s,r)+1$, which is impossible.
\begin{Subcase}
$(G[B],c)$ contains no blue edge. 
\end{Subcase}
By condition 2, there is no blue $K_3$ in $(G[V_1\cup T],c)$. Suppose that $(G[V_1\cup T],c)$ contains red edges. If $(G[R],c)$ contains no red edge, then by condition 2, there is no red $K_3$ in $(G[V_1\cup T],c)$. By induction, $|V_1\cup T|\le f(k,s+2,r-2)$ and $|B|\le|R|\le f(k-1,s,r-1)$. By $(c)$, $|G|=|V_1\cup T|+|R|+|B|\le f(k,s+2,r-2)+2f(k-1,s,r-1)<f(k,s,r)+1$, a contradiction. If $(G[R],c)$ contains red edges, then by condition 4, there can not be red $P_3$ in either $G[V_1\cup T],c)$ or $(G[R],c)$. By induction, $|V_1\cup T|\le f(k,s+1,r-2)$ and $|R|\le f(k,s,r-1)$. Recall that $|B\cap Y_2|\le3$. By $(*)$, $(c)$ and facts (6), (7), we have $|G|\le f(k,s+1,r-2)+f(k,s,r-1)+\{3f(k-2,s,r-2); 2f(k-1,s,r-2)\}<f(k,s,r)+1$, contrary to the fact that $|G|=f(k,s,r)+1$. By fact (7), we now consider the case that $|B|\le f(k,s+1,r-2)+f(k-2,s,r-2)$. As $(G[R],c)$ contains no red $P_3$, we have the following claim.
\begin{claim}\label{c6}
$|R|\le f(k,s+1,r-2)+max\{f(k-1,s,r-2), 4\}$.
\end{claim}
\begin{proof}
Note that $|R\cap Y_2|\le3$. Similar to the proof of facts (6) and (7), we have $|R|\le 3f(k-1,s,r-2)$ when $|R\cap Y_2|=3$, $|R|\le |V_1|+4\le f(k,s+1,r-2)+4$ when $|R\cap Y_2|\le1$ and $|R|\le 2f(k-1,s,r-2)+2$ when $|R\cap Y_2|=2$ and $|R\cap Y_1|\ge1$. Furthermore, if  $|R\cap Y_2|=2$ and $|R\cap Y_1|=0$, then we have $|R|\le 2f(k,s,r-2)=2f(k-1,s,r-2)$ or $|R|\le |V_1|+f(k-1,s,r-2)\le f(k,s+1,r-2)+f(k-1,s,r-2)$. By Table 2, we know that $2f(k-1,s,r-2)\le f(k,s+1,r-2)$. Therefore, $|R|\le 3f(k-1,s,r-2)\le f(k,s+1,r-2)+max\{f(k-1,s,r-2), 4\}$.
\end{proof}
\noindent
Recall that $|V_1\cup T|\le f(k,s+1,r-2)$ and $|B|\le f(k,s+1,r-2)+f(k-2,s,r-2)$. By Claim \ref{c6} and $(c)$, we have $|G|=|V_1\cup T|+|R|+|B|\le3f(k,s+1,r-2)+f(k-2,s,r-2)+max\{f(k-1,s,r-2),4\}<f(k,s,r)+1$, contrary to the fact that $|G|=f(k,s,r)+1$.\medskip

We now see that $(G[V_1\cup T],c)$ contains no red edge. By condition 2, there can not be red $K_3$ in $(G[R],c)$. By induction, $|V_1\cup T|\le f(k-1,s+1,r-2)$ and $|R|\le f(k,s+1,r-1)$. Then by $(*)$, $(c)$ and facts (6), (7),  we have $|G|\le f(k-1,s+1,r-2)+f(k,s+1,r-1)+\{3f(k-2,s,r-2); 2f(k-1,s,r-2)\}<f(k,s,r)+1$, contrary to the fact that $|G|=f(k,s,r)+1$. By fact (7), it remains to consider the case that $|B|\le f(k-1,s+1,r-2)+f(k-2,s,r-2)$. As $(G[R],c)$ contains no red $K_3$, we have the following claim.
\begin{claim}\label{c7}
 $|R|\le10f(k-2,s,r-2)$.
\end{claim}
\begin{proof}
Suppose that $R$ has a part of $\{V_2,\ldots,V_{\ell}\}$ with red edges, say $V^{r}$. In order to avoid a red $K_3$, $V^{r}$ has to be blue-adjacent to the rest of parts in $R$. Furthermore, there is neither red nor blue $K_3$ in $(G[R\setminus V^{r}],c)$. So $R$ contains at most $R_2(K_3)=6$ parts of $\{V_2,\ldots,V_{\ell}\}$. Note that $R$ has at least 5 parts of $\{V_2,\ldots,V_{\ell}\}$, for otherwise we have $|R|\le 4|V_1|$. Recall that $|V_1\cup T|\le f(k-1,s+1,r-2)$ and $|B|\le f(k-1,s+1,r-2)+f(k-2,s,r-2)$. So by $(*)$ and $(c)$, we have $|G|\le 6f(k-1,s+1,r-2)+f(k-2,s,r-2)<f(k,s,r)+1$, which is a contradiction. Note that each part of $R\setminus V^{r}$ can not only be red or blue-adjacent to the rest of parts in $R\setminus V^{r}$, for otherwise we obtain a blue or red $K_3$ as $R(K_2,K_3)=2$. This means that all parts of $R\setminus V^{r}$ contain neither red nor blue edge. By induction, $|R|\le |V_1|+5f(k-2,s,r-2)$. Then by $(*)$ and $(c)$, we have $|G|\le 3f(k-1,s+1,r-2)+6f(k-2,s,r-2)<f(k,s,r)+1$, a contradiction.\medskip

Suppose that $R$ has a part of $\{V_2,\ldots,V_{\ell}\}$ with blue edges, say $V^{b}$.  Let $N_b(V^{b})$ and $N_r(V^{b})$ be the vertex sets of $R$ such that $R=V^{b}\cup N_r(V^{b})\cup N_b(V^{b})$ and all the vertices of $N_b(V^{b})$ and $N_r(V^{b})$ are blue and red-adjacent to $V^{b}$, respectively. We see that there can not be red $K_3$ and blue $P_3$ in $(G[N_b(V^{b})],c)$. By induction, $|N_b(V^{b})|\le f(k,s+1,r-2)$.  Also, there can not be red edge in $(G[N_r(V^{b})],c)$. Then $|N_r(V^{b})\cap Y_2|\le3$. Similar to the proof of facts (6) and (7), we have $|N_r(V^{b})|\le 3f(k-2,s,r-2)$ when $|N_r(V^{b})\cap Y_2|=3$ and $|N_r(V^{b})|\le |V_1|+f(k-2,s,r-2)$ or $|N_r(V^{b})|\le 2f(k-1,s,r-2)$ when $|N_r(V^{b})\cap Y_2|\le2$. It follows that $|R|\le |V_1|+3f(k-2,s,r-2)+f(k,s+1,r-2)$ when $|N_r(V^{b})\cap Y_2|=3$ and $|R|\le2|V_1|+f(k-2,s,r-2)+f(k,s+1,r-2)$ or $|R|\le|V_1|+2f(k-1,s,r-2)+f(k,s+1,r-2)$ when $|N_r(V^{b})\cap Y_2|\le2$. By $(*)$ and $(c)$, we have $|G|\le 3f(k-1,s+1,r-2)+f(k,s+1,r-2)+\{4f(k-2,s,r-2); f(k-1,s+1,r-2)+2f(k-2,s,r-2); 2f(k-1,s,r-2)+f(k-2,s,r-2)\}<f(k,s,r)+1$, which are impossible.\medskip

By the proof above, we see that all the parts of $\{V_2,\ldots,V_{\ell}\}$ in $R$ containing neither red nor blue edge. Note that there are at most $R(K_3,H_3)-1=10$ parts of $\{V_2,\ldots,V_{\ell}\}$ in $R$. So by induction, we have $|R|\le10f(k-2,s,r-2)$.
\end{proof}
\noindent
Recall that $|V_1\cup T|\le f(k-1,s+1,r-2)$ and $|B|\le f(k-1,s+1,r-2)+f(k-2,s,r-2)$. Then by Claim \ref{c7}, $(*)$ and $(c)$, we have $|G|\le 2f(k-1,s+1,r-2)+11f(k-2,s,r-2)<f(k,s,r)+1$, contrary to the fact that $|G|=f(k,s,r)+1$.\medskip

This completes the proof of Theorems \ref{t1.3} and \ref{t1.4}.

\newpage
\section*{Appendix}
Note that the sequence for all cases in each of the following functions is consistent with $w(k,r)$ or $f(k,s,r)$.
\[w(k-1,r)=~~~~~~~~~~~~~~~~~~~w(k-2,r)=~~~~~~~~~~~~~~~~~~~~w(k,r-1)=~~~~~~~~~~~~~~\]
\[\begin{cases}
\frac{2}{5}\cdot5^{(k-r)/2}\cdot 17^{r/2},\\
5^{(k-r-1)/2}\cdot17^{r/2},\\
4\cdot5^{(k-r-1)/2}\cdot17^{(r-1)/2},\\
\frac{8}{5}\cdot5^{(k-r)/2}\cdot17^{(r-1)/2},
\end{cases}
\begin{cases}
\frac{1}{5}\cdot5^{(k-r)/2}\cdot17^{r/2},\\
\frac{2}{5}\cdot5^{(k-r-1)/2}\cdot17^{r/2},\\
\frac{8}{5}\cdot5^{(k-r-1)/2}\cdot17^{(r-1)/2},\\
\frac{4}{5}\cdot5^{(k-r)/2}\cdot 17^{(r-1)/2},
\end{cases}
\begin{cases}
\frac{8}{17}\cdot5^{(k-r)/2}\cdot17^{r/2},\\
\frac{20}{17}\cdot5^{(k-r-1)/2}\cdot17^{r/2},\\
5\cdot5^{(k-r-1)/2}\cdot17^{(r-1)/2},\\
2\cdot5^{(k-r)/2}\cdot 17^{(r-1)/2},
\end{cases}\]

\[w(k-1,r-1)=~~~~~~~~~~~~~~~w(k-2,r-1)=~~~~~~~~~~~~w(k,r-2)=~~~~~~~~~~~~~~~\]
\[\begin{cases}
\frac{4}{17}\cdot5^{(k-r)/2}\cdot17^{r/2},\\
\frac{8}{17}\cdot5^{(k-r-1)/2}\cdot17^{r/2},\\
2\cdot5^{(k-r-1)/2}\cdot17^{(r-1)/2},\\
5^{(k-r)/2}\cdot17^{(r-1)/2}, 
\end{cases}
\begin{cases}
\frac{8}{85}\cdot5^{(k-r)/2}\cdot17^{r/2},\\
\frac{4}{17}\cdot5^{(k-r-1)/2}\cdot17^{r/2},\\
5^{(k-r-1)/2}\cdot17^{(r-1)/2},\\
\frac{2}{5}\cdot5^{(k-r)/2}\cdot 17^{(r-1)/2},
\end{cases}
\begin{cases}
\frac{5}{17}\cdot5^{(k-r)/2}\cdot17^{r/2},\\
\frac{10}{17}\cdot5^{(k-r-1)/2}\cdot 17^{r/2},\\
\frac{40}{17}\cdot5^{(k-r-1)/2}\cdot17^{(r-1)/2},\\
\frac{20}{17}\cdot5^{(k-r)/2}\cdot17^{(r-1)/2},
\end{cases}\]

\[w(k-1,r-2)=~~~~~~~~~~~~~~~w(k-2,r-2)=~~~~~~~~~~~\]
\[\begin{cases}
\frac{2}{17}\cdot5^{(k-r)/2}\cdot 17^{r/2},\\
\frac{5}{17}\cdot5^{(k-r-1)/2}\cdot17^{r/2},\\
\frac{20}{17}\cdot5^{(k-r-1)/2}\cdot17^{(r-1)/2},\\
\frac{8}{17}\cdot5^{(k-r)/2}\cdot17^{(r-1)/2},
\end{cases}
\begin{cases}
\frac{1}{17}\cdot5^{(k-r)/2}\cdot 17^{r/2},\\
\frac{2}{17}\cdot5^{(k-r-1)/2}\cdot17^{r/2},\\
\frac{8}{17}\cdot5^{(k-r-1)/2}\cdot17^{(r-1)/2},\\
\frac{4}{17}\cdot5^{(k-r)/2}\cdot17^{(r-1)/2}.
\end{cases}\]

\[f(k-1,s-1,r)=~~~~~~~~~~~~~f(k-2,s-2,r)=~~~~~~~~~~~~~~~~~f(k,s+1,r-1)=~~~~~~~~~~~~\]
\[\begin{cases}
\frac{2}{5}\cdot5^{s/2}\cdot\lfloor21\cdot17^{(r-2)/2}\rfloor,\\
\frac{4}{5}\cdot5^{s/2}\cdot17^{r/2},\\
5^{(s-1)/2}\cdot\lfloor21\cdot17^{(r-2)/2}\rfloor,\\
2\cdot5^{(s-1)/2}\cdot17^{r/2},\\
4\cdot5^{(s-1)/2}\cdot17^{(r-1)/2},\\
\lfloor\frac{32}{5}\cdot5^{(s-1)/2}\rfloor\cdot17^{(r-1)/2},\\
2\cdot5^{s/2}\cdot17^{(r-1)/2},\\
16\cdot5^{(s-2)/2}\cdot17^{(r-1)/2},
\end{cases}
\begin{cases}
\frac{1}{5}\cdot5^{s/2}\cdot\lfloor21\cdot17^{(r-2)/2}\rfloor,\\
\frac{2}{5}\cdot5^{s/2}\cdot17^{r/2},\\
\frac{1}{5}\cdot5^{(s-1)/2}\cdot\lfloor42\cdot17^{(r-2)/2}\rfloor,\\
\frac{4}{5}\cdot5^{(s-1)/2}\cdot17^{r/2},\\ 
2\cdot5^{(s-1)/2}\cdot17^{(r-1)/2},\\
\frac{16}{5}\cdot5^{(s-1)/2}\cdot17^{(r-1)/2},\\
\frac{4}{5}\cdot5^{s/2}\cdot17^{(r-1)/2},\\
\lfloor\frac{32}{5}\cdot5^{(s-2)/2}\rfloor\cdot17^{(r-1)/2},
\end{cases}
\begin{cases}
10\cdot5^{s/2}\cdot17^{(r-2)/2},\\
\frac{16}{17}\cdot5^{s/2}\cdot17^{r/2},\\
20\cdot5^{(s-1)/2}\cdot17^{(r-2)/2},\\
\frac{32}{17}\cdot5^{(s-1)/2}\cdot17^{r/2},\\
5\cdot5^{(s-1)/2}\cdot\lfloor\frac{21}{17}\cdot17^{(r-1)/2}\rfloor,\\
10\cdot5^{(s-1)/2}\cdot17^{(r-1)/2},\\
5^{s/2}\cdot\lfloor\frac{42}{17}\cdot17^{(r-1)/2}\rfloor,\\
20\cdot5^{(s-2)/2}\cdot17^{(r-1)/2},
\end{cases}\]
\newpage
\[f(k-1,s,r-1)=~~~~~~~~~~~~~~f(k,s,r-1)=~~~~~~~~~~~~~~~f(k-1,s-1,r-1)=~~~~~~~\]
\[\begin{cases}
4\cdot5^{s/2}\cdot17^{(r-2)/2},\\
\frac{1}{17}\cdot\lfloor\frac{32}{5}\cdot5^{s/2}\rfloor\cdot17^{r/2},\\
10\cdot5^{(s-1)/2}\cdot17^{(r-2)/2},\\
\frac{16}{17}\cdot5^{(s-1)/2}\cdot17^{r/2},\\
5^{(s-1)/2}\cdot\lfloor\frac{42}{17}\cdot17^{(r-1)/2}\rfloor,\\
4\cdot5^{(s-1)/2}\cdot17^{(r-1)/2},\\
5^{s/2}\cdot\lfloor\frac{21}{17}\cdot17^{(r-1)/2}\rfloor,\\
10\cdot5^{(s-2)/2}\cdot17^{(r-1)/2}, 
\end{cases}
\begin{cases}
\lfloor\frac{32}{5}\cdot5^{s/2}\rfloor\cdot17^{(r-2)/2},\\
\frac{1}{17}\cdot\lfloor\frac{32}{5}\cdot5^{s/2}\rfloor\cdot17^{r/2},\\
16\cdot5^{(s-1)/2}\cdot17^{(r-2)/2},\\
\frac{16}{17}\cdot5^{(s-1)/2}\cdot17^{r/2},\\
4\cdot5^{(s-1)/2}\cdot17^{(r-1)/2},\\
4\cdot5^{(s-1)/2}\cdot17^{(r-1)/2},\\
2\cdot5^{s/2}\cdot17^{(r-1)/2},\\
10\cdot5^{(s-2)/2}\cdot17^{(r-1)/2},
\end{cases}
\begin{cases}
\frac{16}{5}\cdot5^{s/2}\cdot17^{(r-2)/2},\\
\frac{16}{85}\cdot5^{s/2}\cdot17^{r/2},\\
\lfloor\frac{32}{5}\cdot5^{(s-1)/2}\rfloor\cdot17^{(r-2)/2},\\
\frac{1}{17}\cdot\lfloor\frac{32}{5}\cdot5^{(s-1)/2}\rfloor\cdot17^{r/2},\\
2\cdot5^{(s-1)/2}\cdot17^{(r-1)/2},\\
2\cdot5^{(s-1)/2}\cdot17^{(r-1)/2},\\
\frac{4}{5}\cdot5^{s/2}\cdot17^{(r-1)/2},\\
4\cdot5^{(s-2)/2}\cdot17^{(r-1)/2},
\end{cases}\]

\[~f(k-2,s-1,r-1)=~~~~~~~~~f(k,s+2,r-2)=~~~~~~~~~~~~~~~~f(k,s+1,r-2)=~~~~~~~~~\]
\[\begin{cases}
2\cdot5^{s/2}\cdot17^{(r-2)/2},\\
\frac{16}{85}\cdot5^{s/2}\cdot17^{r/2},\\
4\cdot5^{(s-1)/2}\cdot17^{(r-2)/2},\\
\frac{1}{17}\cdot\lfloor\frac{32}{5}\cdot5^{(s-1)/2}\rfloor\cdot17^{r/2},\\
5^{(s-1)/2}\cdot\lfloor\frac{21}{17}\cdot17^{(r-1)/2}\rfloor,\\
2\cdot5^{(s-1)/2}\cdot17^{(r-1)/2},\\
\frac{1}{5}\cdot5^{s/2}\cdot\lfloor\frac{42}{17}\cdot17^{(r-1)/2}\rfloor,\\
4\cdot5^{(s-2)/2}\cdot17^{(r-1)/2},
\end{cases}
\begin{cases}
5\cdot5^{s/2}\cdot\lfloor\frac{21}{17}\cdot17^{(r-2)/2}\rfloor,\\
\frac{10}{17}\cdot5^{s/2}\cdot17^{r/2},\\
5\cdot5^{(s-1)/2}\cdot\lfloor\frac{42}{17}\cdot17^{(r-2)/2}\rfloor,\\
\frac{20}{17}\cdot5^{(s-1)/2}\cdot17^{r/2},\\
\frac{50}{17}\cdot5^{(s-1)/2}\cdot17^{(r-1)/2},\\
\frac{80}{17}\cdot5^{(s-1)/2}\cdot17^{(r-1)/2},\\
\frac{20}{17}\cdot5^{s/2}\cdot17^{(r-1)/2},\\
\frac{160}{17}\cdot5^{(s-2)/2}\cdot17^{(r-1)/2}, 
\end{cases}
\begin{cases}
4\cdot5^{s/2}\cdot17^{(r-2)/2},\\
\frac{4}{17}\cdot5^{s/2}\cdot17^{r/2},\\
10\cdot5^{(s-1)/2}\cdot17^{(r-2)/2},\\
\frac{10}{17}\cdot5^{(s-1)/2}\cdot17^{r/2},\\
\frac{32}{17}\cdot5^{(s-1)/2}\cdot17^{(r-1)/2},\\
\frac{32}{17}\cdot5^{(s-1)/2}\cdot17^{(r-1)/2},\\
\frac{16}{17}\cdot5^{s/2}\cdot17^{(r-1)/2},\\
\frac{80}{17}\cdot5^{(s-2)/2}\cdot17^{(r-1)/2},
\end{cases}\]

\[f(k-1,s+1,r-2)=~~~~~~~~~~f(k-2,s,r-2)=~~~~~~~~~~~f(k,s,r-2)=f(k-1,s,r-2)=\]
\[\begin{cases}
5^{s/2}\cdot\lfloor\frac{42}{17}\cdot17^{(r-2)/2}\rfloor,\\
\frac{4}{17}\cdot5^{s/2}\cdot17^{r/2},\\
5\cdot5^{(s-1)/2}\cdot\lfloor\frac{21}{17}\cdot17^{(r-2)/2}\rfloor,\\
\frac{10}{17}\cdot5^{(s-1)/2}\cdot17^{r/2},\\
\frac{20}{17}\cdot5^{(s-1)/2}\cdot17^{(r-1)/2},\\
\frac{32}{17}\cdot5^{(s-1)/2}\cdot17^{(r-1)/2},\\
\frac{10}{17}\cdot5^{s/2}\cdot17^{(r-1)/2},\\
\frac{80}{17}\cdot5^{(s-2)/2}\cdot17^{(r-1)/2}, 
\end{cases}
\begin{cases}
5^{s/2}\cdot\lfloor\frac{21}{17}\cdot17^{(r-2)/2}\rfloor,\\
\frac{2}{17}\cdot5^{s/2}\cdot17^{r/2},\\
5^{(s-1)/2}\cdot\lfloor\frac{42}{17}\cdot17^{(r-2)/2}\rfloor,\\
\frac{4}{17}\cdot5^{(s-1)/2}\cdot17^{r/2},\\
\frac{10}{17}\cdot5^{(s-1)/2}\cdot17^{(r-1)/2},\\
\frac{16}{17}\cdot5^{(s-1)/2}\cdot17^{(r-1)/2},\\
\frac{4}{17}\cdot5^{s/2}\cdot17^{(r-1)/2},\\
\frac{1}{17}\cdot\lfloor32\cdot5^{(s-2)/2}\rfloor\cdot17^{(r-1)/2},
\end{cases}
\begin{cases}
2\cdot5^{s/2}\cdot17^{(r-2)/2},\\
\frac{2}{17}\cdot5^{s/2}\cdot17^{r/2},\\
4\cdot5^{(s-1)/2}\cdot17^{(r-2)/2},\\
\frac{4}{17}\cdot5^{(s-1)/2}\cdot17^{r/2},\\
\frac{16}{17}\cdot5^{(s-1)/2}\cdot17^{(r-1)/2},\\
\frac{16}{17}\cdot5^{(s-1)/2}\cdot17^{(r-1)/2},\\
\frac{1}{17}\cdot\lfloor\frac{32}{5}\cdot5^{s/2}\rfloor\cdot17^{(r-1)/2},\\
\frac{1}{17}\cdot\lfloor32\cdot5^{(s-2)/2}\rfloor\cdot17^{(r-1)/2}.
\end{cases}\]

\end{document}